\documentclass[12pt]{article}
\input{amssym}

\newtheorem{prethm}{{\bf Theorem}}

\newenvironment{thm}{\begin{prethm}{\hspace{-0.5
em}{\bf.}}}{\end{prethm}}
\newtheorem{precor}{{\bf Corollary}}

\newenvironment{cor}{\begin{precor}{\hspace{-0.5
em}{\bf.}}}{\end{precor}}
\newtheorem{preprop}{{\bf Proposition}}

\newtheorem{preque}{{\bf Question}}

\newtheorem{preques}{{\bf Question}}

\newtheorem{prelemma}{{\bf Lemma}}

\newenvironment{lemma}{\begin{prelemma}{\hspace{-0.5
em}{\bf.}}}{\end{prelemma}}
\newtheorem{prefact}{{\bf Fact}}

\newtheorem{preobs}{{\bf Observation}}

\newtheorem{prefig}{{\bf Figure}}

\newtheorem{prelemm}{{\bf Lemma}}

\newtheorem{preex}{{\bf Example}}

\newenvironment{ex}{\begin{preex}{\hspace{-0.5em}{\bf .}}}{\end{preex}}
\newtheorem{prepro}{{\bf Proposition}}

\newtheorem{prelem}{{\bf Theorem}}

\newenvironment{lem}{\begin{prelem}{\hspace{-0.5
em}{\bf.}}}{\end{prelem}}
\newtheorem{preproof}{{\bf Proof.}}

\newenvironment{proof}[1]{\begin{preproof}{\rm
               #1}\hfill{$\rule{2mm}{2mm}$}}{\end{preproof}}

\newtheorem{preconj}{{\bf Conjecture}}

\newtheorem{predeff}{{\bf Definition}}

\newenvironment{deff}{\begin{predeff}{\hspace{-0.5
em}{\bf.}}}{\end{predeff}}
\def\newpic#1{}
\date{}
\begin{document}

\title{
{\Large{\bf Relations between  Metric Dimension and Domination Number of Graphs}}}
%

{\small
\author{
{\sc Behrooz Bagheri Gh.}, {\sc Mohsen Jannesari},
{\sc Behnaz Omoomi}\\
[1mm]
{\small \it  Department of Mathematical Sciences}\\
{\small \it  Isfahan University of Technology} \\
{\small \it 84156-83111, \ Isfahan, Iran}}




 \maketitle \baselineskip15truept

\begin{abstract}
 A set $W\subseteq V(G)$ is called a  resolving set, if
for each two distinct vertices $u,v\in V(G)$ there exists $w\in W$
such that $d(u,w)\neq d(v,w)$,  where $d(x,y)$ is the distance
between the vertices $x$ and $y$. The minimum cardinality of a
resolving set for $G$ is called the  metric dimension of $G$, and
denoted by $\beta(G)$. In  this paper, we prove that in a
connected graph $G$ of order $n$,  $\beta(G)\leq n-\gamma(G)$,
where $\gamma(G)$ is the domination number of $G$, and the
equality holds if and only if $G$ is a complete graph or a
complete bipartite graph $K_{s,t}$, $ s,t\geq 2$.   Then, we
obtain new bounds for $\beta(G)$ in terms of minimum and maximum
degree of $G$.
\end{abstract}

\section{Introduction}
 Throughout the paper, $G=(V,E)$ is a
finite, simple, and connected  graph of order $n$. The distance
between two vertices $u$ and $v$, denoted by $d(u,v)$, is the
length of a shortest path between $u$ and $v$ in $G$. The diameter
of $G$, denoted by $diam(G)$ is $\max\{d(u,v)\ |\ u,v\in V\}$. The
set of all neighbors of a vertex $v$ is denoted by $N(v)$. The
maximum degree and minimum degree of  graph $G$, are denoted by
$\Delta(G)$ and $\delta(G)$, respectively. The notations $u\sim v$
and $u\nsim v$ denote the adjacency and non-adjacency relations
between $u$ and $v$, respectively.

\par
For an ordered set $W=\{w_1,w_2,\ldots,w_k\}\subseteq V(G)$ and a
vertex $v$ of $G$, the  $k$-vector
$$r(v|W):=(d(v,w_1),d(v,w_2),\ldots,d(v,w_k))$$
is called  the {\it metric representation}  of $v$ with respect
to $W$. The set $W$ is called a {\it resolving set} for $G$ if
distinct vertices have different representations.  A resolving
set for $G$ with minimum cardinality is  called a {\it metric
basis}, and its cardinality is the {\it metric dimension} of $G$,
denoted by $\beta(G)$.


It is obvious that to see  whether a given set $W$ is a resolving
set, it is sufficient to consider the
 vertices in $V(G)\backslash W$, because $w\in
W$ is the unique vertex of $G$ for which $d(w,w)=0$.
 When $W$ is a resolving set for $G$, we say
that $W$ {\it resolves} $G$. In general, we say an ordered set $W$
resolves a set $T\subseteq V(G)$ of vertices in $G$, if for each
two distinct vertices $u,v\in T$, $r(u|W)\neq r(v|W)$.

\par In~\cite{Slater1975}, Slater introduced the idea of a resolving
set and used a {\it locating set} and the {\it location number}
for what we call a resolving set and the metric dimension,
respectively. He described the usefulness of these concepts when
working with U.S. Sonar and Coast Guard Loran stations.
Independently, Harary and Melter~\cite{Harary} discovered the
concept of the location number as well and called it the metric
dimension. For more results related to these concepts
see~\cite{cartesian product,bounds,sur1,landmarks,sur2}. The
concept of a resolving set has various applications in diverse
areas including coin weighing problems~\cite{coin}, network
discovery and verification~\cite{net2}, robot
navigation~\cite{landmarks}, mastermind game~\cite{cartesian
product}, problems of pattern recognition and image
processing~\cite{digital}, and combinatorial search and
optimization~\cite{coin}.

\par

The following bound is the known upper bound for metric dimension.

\begin{lem}~{\rm\cite{Ollerman}}\label{n-d}
If $G$ is a connected graph of order $n$, then $\beta(G)\leq
n-diam(G)$.
\end{lem}
%
A set $\Gamma\subseteq V(G)$ is a {\it dominating set} for $G$ if every vertex not in $\Gamma$ has a neighbor in $\Gamma$.
 A dominating set with minimum size is a {\it minimum dominating set} for $G$.
 The {\it domination number} of $G$, $\gamma(G)$,  is the cardinality of a minimum dominating set.
 In  Section~\ref{main}, we prove that $\beta(G)\leq n-\gamma(G)$.
 Moreover, we prove that  $\beta(G)= n-\gamma(G)$ if and only if $G$ is a complete
 graph or a complete bipartite graph $K_{s,t}$, $s,t \geq 2$.
  In Section~\ref{newbounds},  regarding to known bounds of
  $\gamma(G)$, we obtain new upper bounds for metric dimension  in
  terms of other graph parameters.

\section{Main Results}\label{main}

In this section, we prove that $\beta(G)\leq n-\gamma(G)$.
Moreover, we show that  $\beta(G)= n-\gamma(G)$ if and only if $G$
is a complete
 graph or a complete bipartite graph $K_{s,t}$, $s,t \geq 2$.

Two vertices $u,v\in V(G)$ are called {\it false twin} vertices
if $N(u)=N(v)$.

\begin{lemma}\label{lem:domsetnoft}
Let $G$ be a connected graph. Then there exists a minimum dominating set for $G$ which does not have any pair of false twin vertices.
\end{lemma}
\begin{proof}{
Let $\Gamma$ be a minimum dominating set for $G$ with minimum
number of false twin pairs of vertices and $u, v$ be an arbitrary
false twin pair in $\Gamma$. Since $u$ and $v$ dominate the same
vertices in $G$, they have no neighbors in $\Gamma$; otherwise,
$\Gamma\setminus\{u\}$ and $\Gamma\setminus\{v\}$ are dominating
sets in $G$ which is a contradiction. On the hand,    $G$ is
connected, hence $u$ and $v$ have some neighbors  in
$V(G)\setminus \Gamma$.
  Now,  $\Gamma'=\Gamma\cup\{x\}\setminus\{u\}$, where $x$ is a neighbor of $u$ in $V(G)\setminus
  \Gamma$,   is a dominating set for $G$ with fewer  number of  false twin pair of vertices.
 This contradiction implies that $\Gamma$ has no  false twin pair of  vertices.
}\end{proof}

\begin{thm}\label{thm:V-Gama resolv}
For every connected graph $G$ of order $n$, $\beta(G)\leq
n-\gamma(G).$ In particular,  if $\Gamma$ is a minimum dominating
set for $G$ with no false twin pair of  vertices,
 then $V(G)\setminus\Gamma$ is a resolving set for $G$.
\end{thm}

\begin{proof}{
By Lemma~\ref{lem:domsetnoft}, $G$ has a minimum dominating set
$\Gamma$ with no pair of false twin vertices.  Suppose, on the
contrary, that $V(G)\setminus\Gamma$ is not a resolving set for
$G$. Then, there exist vertices $u$ and $v$ in $\Gamma$ such that
$r(u|V(G)\setminus\Gamma)=r(v|V(G)\setminus\Gamma)$. This implies
that all neighbors of $u$ and $v$ in $V(G)\setminus\Gamma$ are the
same. Therefore, $u$ and $v$ have no neighbor in $\Gamma$;
otherwise we can remove one of the vertices $u$ and $v$ from
 $\Gamma$ and get a dominating set with cardinality $|\Gamma|-1$.
  Hence, $u$ and $v$ are false twin vertices, which is a contradiction.
  Thus, $V(G)\setminus\Gamma$ is a resolving set for $G$.
  Accordingly, $\beta(G)\leq n-\gamma(G).$
  }\end{proof}
%
The following example shows that Theorem~\ref{thm:V-Gama resolv}
gives a better upper bound for $\beta(G)$ comparing the upper
bound in Theorem~\ref{n-d}.
\begin{ex}
Let $G$ be a connected graph of order $3k+1$, $k\geq 6$, obtained
from the wheel $W_k$ by replacing each spoke by a path of length
three. It is easy to see that $\gamma(G)=k+1$, by
Theorem~\ref{thm:V-Gama resolv},  $\beta(G) \leq n-\gamma(G)=2k$
while  $diam(G)\le 6$ and by Theorem~\ref{n-d}, $\beta(G)\leq
3k+1-diam(G)$.

\end{ex}

In the sequel we need the following definition.
\begin{deff}\label{def:specific}
Let $\Gamma$ be a dominating set in a connected graph $G$ and
$u\in\Gamma$. A vertex $\bar{u}\in V(G)\setminus\Gamma$ is called
a {\it private neighbor } of  $u$ if $u$ is the unique neighbor of
$\bar{u}$ in $\Gamma$, i.e., $N(\bar{u})\cap\Gamma=\{u\}$.
\end{deff}
It is clear that each vertex of a minimum dominating set $\Gamma$
for a graph $G$ has a private neighbor  or it is a single vertex
in $\Gamma$. The following lemma provides a minimum dominating set
$\Gamma$ for $G$ with no  false twin pair of vertices such that
every vertex in $\Gamma$ has a private neighbor.
\begin{lemma}\label{lem:all have specific}
Every connected graph $G$ has  a minimum dominating set  $\Gamma$
with no false twin  pair of vertices  such that
 every vertex in  $\Gamma$ has a private neighbor.
\end{lemma}

\begin{proof}{
By Lemma~\ref{lem:domsetnoft},
let $\Gamma$ be a minimum dominating set  with  no  false twin
pair of vertices with minimum number of single vertices. Also,
let $u$ be a single vertex in $\Gamma$.
 Since $G$ is a
connected graph, $u$ has  a neighbor in $V(G)\setminus\Gamma$, say
$x$. Now $\Gamma'=\Gamma\cup\{x\}\setminus\{u\}$ is also a minimum
dominating set for $G$  with no false twin pair of vertices,
because $x$ is the unique vertex in  $\Gamma'$ that is adjacent to
$u$. Moreover, $u$ is a private neighbor of  $x$ in
$V(G)\setminus\Gamma'$. Note that, $x$ was not  a private neighbor
of any vertex  in $\Gamma$.
Therefore, the number of vertices in $\Gamma'$ which have a
private neighbor  is more than the number of vertices in $\Gamma$
which have a private neighbor in $V(G)\setminus\Gamma$. On the
other words, the number of single vertex in  $\Gamma'$ is fewer
than $\Gamma$.  This contradiction implies that all vertices in
$\Gamma$ have a private neighbor in $V(G)\setminus\Gamma$.
}\end{proof}
\begin{thm}\label{thm:character}
Let $G$ be a connected graph of order $n$. Then
$\beta(G)=n-\gamma(G)$ if and only if $G=K_n$ or $G=K_{s,t}$, for
some $s,t\ge2$.
\end{thm}

\begin{proof}{
Clearly, for $G=K_n$ and
 $G=K_{s,t}$, $s,t\ge2$, the equality holds.  Now let $\beta(G)=n-\gamma(G)$.  By
Lemma~\ref{lem:all have specific}, there exists a minimum
dominating set $\Gamma$ for $G$ with no false twin vertices such
that all vertices in $\Gamma$ have a private neighbor  in
$V(G)\setminus \Gamma$.  Let $\Gamma=\{u_1,u_2,\ldots,u_r\}$ and
$W_1=\{x_1,x_2,\ldots,x_r\}$, where  $x_i$ is private neighbor of
$u_i$ for an $i$, $1\leq i\leq r$. Since $u_i$ is the unique
neighbor of $x_i$ in $\Gamma$, for each $i,j$, $1\leq i,j\leq r$,
the $i^{th}$ coordinate of $r(u_j|W_1)$ is $1$ if and only if
$j=i$. Therefore, $W_1$ resolves the set $\Gamma$.
\par
 By Theorem~\ref{thm:V-Gama resolv},
 $W=V(G)\setminus \Gamma$ is a resolving set for $G$ and  $\beta(G)=n-\gamma(G)$ implies that $W$ is a metric basis. Now let
$x\in W\setminus W_1$. Since $W_1$ resolves $\Gamma$, there exists
a unique vertex $u_i\in \Gamma$ such that $r(x|W_1)=r(u_i|W_1)$.
Thus, $x$ and $u_i$ have the same neighbors in $W_1$, but
$N(u_i)\cap W_1=\{x_i\}$, hence $N(x)\cap W_1=\{x_i\}$. Thus,
$W\setminus W_1$ is partitioned into sets $V_1,V_2,\ldots,V_r$,
(some $V_i$'s could be empty) such that for each $i$, $1\leq i\leq
r$, and every $x\in V_i$, $N(x)\cap W_1=\{x_i\}$. Therefore, $W_1$
is a minimum dominating set for $G$. Moreover, $W_1$ has no pair
of false twin vertices,
  because for each $i$, $1\leq i\leq r$, $x_i$ is the unique neighbor of  $u_i$ in $W_1$.
  Hence, by Theorem~\ref{thm:V-Gama resolv} the set $B=V(G)\setminus W_1$ is a metric basis of $G$.
\par
\par
Now let $B_i=V_i\cup\{u_i\}$.   For a fixed $i$, $1\leq i\leq r$,
let $a$ be an arbitrary vertex in  $B_i$. Since $B$ is a metric
basis of $G$, $B\setminus\{a\}$ is not a resolving set for $G$.
Therefore, there exists a vertex $x_{j_a}\in W_1$ such that
$r(a|B\setminus\{a\})=r(x_{j_a}|B\setminus\{a\})$. If $j_a=i$,
then $a$ is adjacent to all vertices in $B_i\setminus\{a\}$, since
$x_i$ is adjacent to all vertices in $B_i$. If $j_a\neq i$, then
$a$ is not adjacent to any vertex in $B_i$, since $x_j$, $j\neq
i$, is not adjacent to any vertex in $B_i$. Hence, for every two
vertices $a$ and $a'$ in $B_i$, where $j_a=j_{a'}$.  Thus, we
conclude that,  for every vertex  $a\in B_i$, there exists a
vertex $x_{j}\in W_1$ such that
$r(a|B\setminus\{a\})=r(x_{j}|B\setminus\{a\})$, and there are two
possibilities $j=i$ or $j\neq i$; in the former case $B_i$ is a
clique and in the latter case $B_i$ is an independent set.

Now let there exists $i$, $1\leq i\leq r$, such that for every
vertex $a\in B_i$,
$r(a|B\setminus\{a\})=r(x_{i}|B\setminus\{a\})$.
 It was shown that in this case $B_i$ is a clique. Moreover, since $a$ is not adjacent to
any vertex in $W_1\setminus\{x_i\}$, $x_i$ is not adjacent to any
vertex in $W_1\setminus\{x_i\}$. Moreover, since $x_i$ is not
adjacent to any vertex in $B\setminus B_i$, $a$ is not adjacent to
any vertex in $B\setminus B_i$.
%
 Therefore, the induced subgraph by $B_i\cup\{x_i\}$  is a  maximal connected subgraph of $G$.
 Since $G$ is a connected graph, $G=B_i\cup\{x_i\}$, and consequently $G=K_n$.
\par
Otherwise,  for each $i$, $1\leq i\leq r$,  and for every vertex
 $a\in B_i$,  $r(a|B\setminus\{a\})\neq
r(x_i|B\setminus\{a\})$ and
$r(a|B\setminus\{a\})=r(x_j|B\setminus\{a\})$ for some $j\neq i$.
Now, for each $b\in B_j$, if $r(b|B\setminus
\{b\})=r(x_k|B\setminus \{b\})$, then $x_k$ is adjacent to all
vertices in $B_i$, since $b$ is adjacent to all vertices in $B_i$.
Thus, $k=i$. It was shown that in this case each $B_i$, $1\leq i
\leq r$, is an independent set. Now, since $x_j$ is adjacent to
all vertices in $B_j$, every vertex $a\in B_i$ is adjacent to all
vertices in $B_j$. Therefore, the induced subgraph $B_i\cup B_j$
is a complete bipartite graph.
\par
Note that, each vertex in $B_i\cup B_j$ is not adjacent to any
vertex in $B_k$, $k\in\{1,2,\dots, r\}\setminus\{i,j\}$,
 because $x_i$ and $x_j$  are not adjacent to any vertex in $B\setminus(B_i\cup B_j)$.
 On the other hand, $x_i$ and $x_j$  are not adjacent to any vertex in $W_1\setminus\{x_i,x_j\}$,
 since all vertices in $B_i\cup B_j$ are not adjacent to any vertex in this set.
  Therefore, the induced subgraph by $B_i\cup B_j\cup \{x_i,x_j\}$  is a  maximal connected subgraph of $G$.
  Since $G$ is a connected graph, $G=B_i\cup B_j\cup \{x_i,x_j\}$.
   Furthermore, $r(a|B\setminus\{a\})=r(x_j|B\setminus\{a\})$ implies that $x_i\sim x_j$, because $a\sim x_i$. Thus, $G=K_{s,t}$.
    Since $u_i\in B_i$ and $u_j\in B_j$, $s,t\geq2$.
}\end{proof}
%
\section{Upper Bounds for the Metric Dimension }\label{newbounds}

The domination number is a well studied parameter and there are
several bounds for $\gamma(G)$ in terms of the other graph
parameters. Following the given new upper bound for $\beta(G)$ in
Theorem~\ref{thm:V-Gama resolv}, several new upper bounds for
metric dimension can be obtained. In what follows, we present some
of these new upper bounds.
\begin{lem}\label{thm:gamma>2(delta-1)}~{\rm\cite{MR1605684}}
For every graph $G$ of order $n$ and girth $g$,
\begin{description}
\item {\rm (i)} if $g\ge 5$, then $\gamma(G)\ge \delta(G)$.
\item {\rm (ii)} if $g\ge 6$, then $\gamma(G)\ge 2(\delta(G)-1)$.
\item {\rm (iii)} $\gamma(G) \geq
\left\lceil\frac{n}{1+\Delta(G)}\right\rceil$.
\item {\rm (iv)} If  $G$ has degree sequence
$(d_1,d_2,\ldots,d_n)$ with $d_i\ge d_{i+1}$, then $\gamma(G)\ge
\min\{k\ | \ k+(d_1+d_2+\cdots+d_k)\ge n\}.$
 \item {\rm (v)} if $\delta(G) \ge 2$ and $g\ge 7$, then
$\gamma(G)\ge \Delta(G).$
\end{description}
\end{lem}
%
%
%
%

\begin{lem}~{\rm\cite{Laplacian}}
Let  $\mu_n\ge \mu_{n-1}\ge \cdots\ge \mu_1$ be the eigenvalues of
Laplacian matrix of connected graph $G$ of order $n\ge 2$, then
$\gamma(G)\ge \frac{n}{\mu_n(G)}.$
\end{lem}
By Theorem~\ref{thm:V-Gama resolv} and  above theorems, the
 list of new upper bounds for metric
dimension in terms of other graph parameters are obtained.

\begin{cor}\label{cor:beta<n-2(delta-1)}
For every connected graph $G$ of order $n$ and girth $g$,
\begin{description}
\item {\rm (i)} if $g\ge 5$, then $\beta(G)\le n-\delta(G)$.
\item {\rm (ii)} if $g\ge 6$, then $\beta(G)\le n-2\delta(G)+2$.
\item {\rm (iii)} $\beta(G)\leq
n(G)-\left\lceil\frac{n}{1+\Delta(G)}\right\rceil.$
\item {\rm (iv)}  if $G$ has degree sequence
$(d_1,d_2,\ldots,d_n)$ with $d_i\ge d_{i+1}$,
 then $\beta(G)\leq n-\min\{k \ | \ k+(d_1+d_2+\cdots+d_k)\ge n\}.$
 \item {\rm (v)} if $\mu_n\ge \mu_{n-1}\ge \cdots\ge \mu_1$ be the eigenvalues of
Laplacian matrix of $G$, then $\beta(G)\le n-\frac{n}{\mu_n(G)}.$
\item {\rm (vi)} if $\delta(G) \ge 2$ and $g\ge 7$, then
$\beta(G)\leq n-\Delta(G).$
\end{description}
\end{cor}

For each of the given upper bounds in above, infinite classes of
graphs can be constructed to show that these bounds could be
better than $n-diam(G)$.

In the following example, we consider the well known graph Kneser
 $KG(2k+1, k)$, which is called  odd graph.  The Kneser
graph with integer parameters $n$ and $k$, $n\geq 2k$, denoted by
$KG(n,k)$, is the graph with  $k$ element subsets of set
$\{1,2,\ldots,n\}$ as the vertex set and  two vertices are
adjacent if and only if the corresponding subsets are disjoint.
\begin{ex}
Let $G=KG(2k+1,k)$, for $k\ge 3$ . Then, $n=|V(G)|={2k+1 \choose
k}$, $\Delta(G)=\delta(G)=k+1$, $g(G)=6$, $\mu_{{2k+1 \choose
k}}(G)=2k+1$, and $diam(G)=k$. Therefore, we have:
\begin{description}
\item {\rm (i)} $\beta(G)\le n-k-1$.
\item {\rm (ii)} $\beta(G)\le n-2k$.
\item {\rm (iii)} $\beta(G)\leq n-\left\lceil\frac{{2k+1 \choose
k}}{k+2}\right\rceil.$
%
\item {\rm (iv)} $\beta(G)\leq n-\frac{{2k+1 \choose k}}{k+2}.$
\item {\rm (v)} $\beta(G)\le n-\frac{{2k+1 \choose k}}{2k+1}.$
\end{description}
\end{ex}
%
%

\end{document}